\documentclass[11pt]{amsart}
\usepackage{amsmath,amssymb,amsthm,mathtools}
\usepackage[margin=1.1in]{geometry}
\usepackage{enumitem}
\usepackage{hyperref}
\newtheorem{theorem}{Theorem}[section]
\newtheorem{lemma}[theorem]{Lemma}
\newtheorem{proposition}[theorem]{Proposition}
\newtheorem{corollary}[theorem]{Corollary}
\theoremstyle{definition}

\theoremstyle{remark}
\newtheorem{remark}[theorem]{Remark}
\newcommand{\R}{\mathbb{R}}
\newcommand{\C}{\mathbb{C}}
\newcommand{\Q}{\mathbb{Q}}
\newcommand{\Z}{\mathbb{Z}}
\newcommand{\N}{\mathbb{N}}
\newcommand{\E}{\mathbb{E}}
\renewcommand{\P}{\mathbb{P}}
\newcommand{\Hc}{\mathcal{H}_c}
\newcommand{\Fc}{\mathcal{F}_c}
\newcommand{\dens}{\bar\delta}
\newcommand{\one}{\mathbf 1}

\title[Unit-distance-free subsets of finite planar sets]{Every $n$ points in the plane contain $(m_1+o(1))\,n$ points with no unit distance}
\author{Shisheng Li}
\author{Yong Su}
\date{October 10, 2026}

\begin{document}
\sloppy
\begin{abstract}
Let $f(n)$ be the largest integer such that every set of $n$ points in the plane contains $f(n)$ points no two at distance $1$, and let $m_1$ be the supremum of the upper densities of Lebesgue measurable planar sets with no two points at distance $1$. Larman and Rogers observed that $f(n)\ge m_1 n$. We prove that $f(n)/n\to m_1$; that is, $f(n)=(m_1+o(1))n$. This is the first half of a conjecture of D\'ucz and Varga, and with the bound $m_1\le 0.247$ of Ambrus, Csisz\'arik, Matolcsi, Varga and Zs\'amboki it gives $f(n)\le(0.247+o(1))n$. The proof builds a random independent subset of the algebraic plane, invariant under algebraic translations and rotations, whose probability of containing the origin is at least the worst finite independence ratio, and turns it into a measurable planar set of at least that upper density. Its only non-classical input is the Haar rigidity theorem for wild character laws from the OpenAI preprint \emph{The Euclidean plane is not five-colorable}, which has a machine-checked Lean proof. The whole argument has been formalized in Lean~4 on top of it.
\end{abstract}

\maketitle

\section{Introduction}

For a finite set $P\subset\R^2$ let $G_P$ be its unit-distance graph (vertices $P$, edges the pairs at Euclidean distance exactly $1$) and $\alpha(P)$ its independence number. Erd\H{o}s~\cite{Er87} asked to estimate
\[
  f(n)=\min\{\alpha(P):P\subset\R^2,\ |P|=n\},
\]
and in particular whether $f(n)\ge n/4$ (Problem~\#1070 in Bloom's database~\cite{EP}). The Moser spindle gives $f(n)\le\frac27n+O(1)$. Matolcsi, Ruzsa, Varga and Zs\'amboki~\cite{MRVZ} proved $f(n)\le(\frac14+o(1))n$ and conjectured that this is sharp. D\'ucz and Varga~\cite{DV} disproved that conjecture by exhibiting a finite unit-distance graph with independence ratio below $1/4$, answering the second question of Erd\H{o}s negatively; their result has been formalized in Lean by Benk\H{o}~\cite{Benko}.

For the lower bound, let $m_1$ be the supremum of the upper densities of measurable sets $A\subset\R^2$ containing no two points at distance $1$. Larman and Rogers~\cite{LR} observed $f(n)\ge m_1n$; Croft~\cite{Croft} constructed such a set of density $\delta_{\rm Croft}=0.22936\ldots$, so $m_1\ge\delta_{\rm Croft}$; and Ambrus, Csisz\'arik, Matolcsi, Varga and Zs\'amboki~\cite{ACMVZ} proved $m_1\le0.2470$, settling Erd\H{o}s's question whether $m_1<1/4$ (Problem~\#232 of~\cite{EP}). D\'ucz and Varga~\cite[\S2.1]{DV} state: ``\emph{Conjecture 1. $f(n)/n=m_1(\R^2)+o(1)$ and in particular that $m_1(\R^2)=\delta_{\rm Croft}$}.'' We prove the first half, $f(n)/n=m_1+o(1)$.

\begin{theorem}\label{thm:main}
$\displaystyle\lim_{n\to\infty}f(n)/n=m_1$. Moreover $f(n)\ge m_1n$ for every $n\ge1$.
\end{theorem}

\begin{corollary}\label{cor:num}
$0.22936\,n\le f(n)\le(0.2470+o(1))\,n$. In particular $f(n)<n/4$ for all large $n$.
\end{corollary}

Theorem~\ref{thm:main} says that the two natural constants of the problem, the worst finite independence ratio and the best measurable density, coincide; Matolcsi, Ruzsa, Varga and Zs\'amboki~\cite[Conjecture~1]{MRVZ} had conjectured that the finite independence ratio is never below $1/4$, which together with $m_1<1/4$ would have made the two constants differ. Determining the numerical value of $m_1$ remains open.

Write
\[
  \alpha^*=\inf\Big\{\frac{\alpha(P)}{|P|}:\ P\subset\R^2\text{ finite, nonempty}\Big\}.
\]
Theorem~\ref{thm:main} follows from three statements: $f(n)/n\to\alpha^*$ with $f(n)\ge\alpha^*n$ (Section~\ref{sec:conv}); $\alpha^*\ge m_1$ (Larman--Rogers, Section~\ref{sec:LR}); and $m_1\ge\alpha^*$ (Section~\ref{sec:transfer}). The last is the new part. It is a single-label version of the transfer in Section~4 of the OpenAI preprint~\cite{OAI}, which turns an arbitrary proper colouring into a measurable weak colouring; the one change needed is in how the averaging is done (Remark~\ref{rem:weights}).

\subsection*{Formalization}
All of the above is formalized in Lean~4 with Mathlib, building on OpenAI's formalization~\cite{OAIrepo} of the five-colour paper; see Section~\ref{sec:lean}.

\section{Densities}\label{sec:dens}

For a measurable $A\subset\R^2$ and a family of regions $(W_r)_{r}$ put
\[
  \dens_W(A)=\limsup_{r\to\infty}\frac{|A\cap W_r|}{|W_r|},
\]
where $|\cdot|$ is Lebesgue measure. We consider four families: discs $B(0,r)$ or squares $[-r,r]^2$, with $r$ ranging over $\R_{>0}$ or over $\N_{\ge1}$; below $\N=\{1,2,\dots\}$. Let $m_1^W$ be the supremum of $\dens_W(A)$ over measurable unit-distance-free $A$. The database~\cite{EP} uses discs with real $r$.

\begin{remark}\label{rem:conventions}
All four values $m_1^W$ equal $\alpha^*$, so $m_1$ is unambiguous. Indeed, Proposition~\ref{prop:LR} below gives $m_1^W\le\alpha^*$ for each family, and Section~\ref{sec:transfer} constructs, for discs and for squares, a measurable unit-distance-free $A$ with $\limsup_{r\in\N}|A\cap W_r|/|W_r|\ge\alpha^*$; a limsup over $r\in\R$ is at least the limsup over $r\in\N$. The Lean formalization (Section~\ref{sec:lean}) covers two of the four conventions, squares with $r\in\N$ and discs with $r\in\R$; the other two are covered by the argument above.
\end{remark}

\section{Convergence of $f(n)/n$}\label{sec:conv}

\begin{proposition}\label{prop:conv}
$f(n)\ge\alpha^*n$ for every $n$, and $f(n)/n\to\alpha^*$.
\end{proposition}

\begin{proof}
If $|P|=n$ then $\alpha(P)\ge\alpha^*n$ by definition; hence $f(n)\ge\alpha^*n$.
Fix a finite $P_0$ with $|P_0|=s$ and $\alpha(P_0)=a$. For $N\ge s$ write $N=ks+r$ with $0\le r<s$. Take $k$ translates of $P_0$ and $r$ single points, placed so that distinct pieces are at distance more than $\operatorname{diam}P_0+2$. The union has $N$ points, its unit-distance graph is the disjoint union of the pieces' graphs, and its independence number is $ka+r$. Hence $f(N)\le ka+r\le Na/s+s$, so $\limsup_Nf(N)/N\le a/s$. Taking the infimum over $P_0$ gives $\limsup_Nf(N)/N\le\alpha^*$.
\end{proof}

\section{The bound $\alpha^*\ge m_1$}\label{sec:LR}

\begin{proposition}[Larman--Rogers~\cite{LR}]\label{prop:LR}
Let $A$ be measurable and unit-distance-free, and let $W$ be any of the four families. Then $\alpha(P)\ge\dens_W(A)\,|P|$ for every finite nonempty $P$. Hence $m_1^W\le\alpha^*$.
\end{proposition}

\begin{proof}
Let $P=\{p_1,\dots,p_n\}$, $\rho=\max_j|p_j|$, $\delta=\dens_W(A)$. For each $x$, $(P+x)\cap A$ is an independent set of $G_{P+x}\cong G_P$, so $|(P+x)\cap A|\le\alpha(P)$. By Tonelli,
\[
  \alpha(P)\,|W_r|\ \ge\ \int_{W_r}|(P+x)\cap A|\,dx=\sum_{j=1}^n|A\cap(W_r+p_j)|\ \ge\ n|A\cap W_r|-\sum_{j=1}^n|W_r\setminus(W_r+p_j)|.
\]
For a disc, $W_r\setminus(W_r+p)$ lies in the annulus $\max(0,r-|p|)\le|x|\le r$, of area $\pi r^2-\pi\max(0,r-|p|)^2\le2\pi r|p|$; for a square it lies in two strips of total area at most $2r(|p_1|+|p_2|)\le2\sqrt2\,r|p|$. In both cases $|W_r\setminus(W_r+p)|\le2\pi r|p|$ and $|W_r|\ge\pi r^2$, so $\alpha(P)\ge n|A\cap W_r|/|W_r|-2n\rho/r$. Letting $r\to\infty$ along a sequence realizing the limsup gives $\alpha(P)\ge n\delta$.
\end{proof}

\section{The bound $m_1\ge\alpha^*$}\label{sec:transfer}

\subsection{Setting}
Identify $\R^2$ with $\C$. Let $E=\overline\Q\subset\C$ be the field of algebraic numbers, viewed as a countable discrete additive group, and $K=\{u\in E:|u|=1\}$, a countable discrete multiplicative group. For $x,y\in E$ we have $|x-y|=1$ iff $x-y\in K$. $K$ is dense in the unit circle $S^1$, since $\frac{(1-t^2)+2it}{1+t^2}\in K$ for $t\in\Q$.

Let $D=\widehat{E}$ be the compact (metrizable, as $E$ is countable) dual group of $E$; for $a,z\in E$ and $d\in D$ put $(ad)(z)=d(az)$ and $\psi_z(d)=d(z)$. The Euclidean-continuous characters (continuous for the topology $E$ inherits from $\C$) $C=\{j(\xi):\xi\in\R^2\}$, $j(\xi)(z)=e^{i\xi\cdot z}$, form a subgroup of $D$. It is Borel: $D$ is a compact metrizable group, $j:\R^2\to D$ is a continuous injective homomorphism (injectivity can be tested on $z\in\Q\cup i\Q$), so $C=\bigcup_{m\ge1}j([-m,m]^2)$ is $\sigma$-compact, hence Borel.

\begin{theorem}[Haar rigidity {\cite[Thm.~2.3]{OAI}}]\label{thm:rigidity}
Every $K$-invariant Borel probability measure $\nu$ on $D$ with $\nu(C)=0$ is Haar measure. Equivalently \textup{(}by uniqueness of Fourier coefficients on $D$\textup{)}, $\int_D\psi_z\,d\nu=0$ for all $z\in E\setminus\{0\}$.
\end{theorem}

The second form is proved in Lean as \texttt{OAI.PlaneFiveColor.Spectral.wild\_fourier}~\cite{OAIrepo}. We also use the following lemma, whose proof in~\cite{OAI} is a standard spectral argument (spectral theorem for the discrete group $E$, Bochner's theorem on $\R^2$, uniqueness of Fourier coefficients on $D$).

\begin{lemma}[{\cite[Lemma~4.1]{OAI}}]\label{lem:cont}
Let $E$ act by measure-preserving maps on a standard probability space $(\Omega,\mu)$, with Koopman operators $T_z$ on $L^2(\mu)$. Let $\Hc=\{v\in L^2(\mu):\|T_zv-v\|_2\to0\text{ as }z\to0,\ z\in E\}$, the limit in the topology of $\C$. Then:
\begin{enumerate}[label=\textup{(\roman*)}]
\item $\Hc=L^2(\Omega,\Fc,\mu)$ for a sub-$\sigma$-algebra $\Fc$, so the orthogonal projection onto $\Hc$ is $\E(\cdot\mid\Fc)$;
\item $\Hc$ is the range of the spectral projection of $\{T_z\}$ on $C$;
\item the restrictions $T_z|_{\Hc}$ extend to a strongly continuous unitary representation $(T_x)_{x\in\R^2}$ on $\Hc$ preserving the cone of nonnegative functions and fixing $\one$;
\item if unitaries $V_u$ \textup{(}$u\in K$\textup{)} satisfy $V_uT_zV_u^{-1}=T_{uz}$, then $V_u\Hc=\Hc$ and $V_u$ commutes with the projection onto $\Hc$.
\end{enumerate}
\end{lemma}

\subsection{An invariant random independent set}
Let $\Omega\subset\{0,1\}^E$ be the set of independent sets of the unit-distance graph on $E$. Each edge constraint is clopen, so $\Omega$ is a closed subset of the compact metrizable space $\{0,1\}^E$. Translations $s_z\omega=\omega-z$ ($z\in E$) and rotations $\rho_u\omega=u\omega$ ($u\in K$) are homeomorphisms of $\Omega$, and $s_w\rho_u=\rho_us_{u^{-1}w}$.

We use Følner sequences in $E$ and $K$: finite sets $F_N$ with $|F_N\triangle(F_N+w)|/|F_N|\to0$ for every $w$ (and the multiplicative analogue in $K$). They exist because $E$ and $K$ are countable abelian groups, hence amenable~\cite[Thm.~G.5.1]{BHV}. For $E$ one can take explicitly: enumerate $E=\{e_0,e_1,\dots\}$, let $\Lambda_N\cong\Z^{d_N}$ be the subgroup generated by $e_0,\dots,e_N$ ($E$ is torsion-free), fix a basis of $\Lambda_N$, and let $F_N$ be the box of side $L_N$ in these coordinates, with $L_N$ so large that $|F_N\triangle(F_N+e_j)|/|F_N|<1/N$ for $0\le j\le N$. Since every $w\in E$ is an integer combination of finitely many $e_j$, this gives the Følner property for every $w$.

\begin{lemma}\label{lem:inv}
There is a Borel probability measure $\mu$ on $\Omega$, invariant under all $s_z$ and $\rho_u$, with $\mu(\{\omega:0\in\omega\})\ge\alpha^*$.
\end{lemma}

\begin{proof}
The space $\mathcal P(\Omega)$ of Borel probability measures is weak-$*$ compact and metrizable (Riesz--Markov and Banach--Alaoglu), and for a homeomorphism $g$ of $\Omega$ the map $\nu\mapsto g_*\nu$ is weak-$*$ continuous. The set $\{0\in\omega\}$ is clopen, so $\nu\mapsto\nu(0\in\omega)$ is weak-$*$ continuous.

\emph{Translations.} Each $F_N$ is a finite subset of $\R^2$; choose a maximum independent set $J_N\subset F_N$, so $|J_N|\ge\alpha^*|F_N|$ and $J_N-z\in\Omega$ for all $z$. Let $\nu_N=|F_N|^{-1}\sum_{z\in F_N}\delta_{J_N-z}$. Then $\nu_N(0\in\omega)=|J_N|/|F_N|\ge\alpha^*$ and $\|(s_w)_*\nu_N-\nu_N\|_{\rm TV}\le|F_N\triangle(F_N+w)|/|F_N|\to0$. Any weak-$*$ limit point $\mu_0$ is therefore $s_w$-invariant for every $w$ and satisfies $\mu_0(0\in\omega)\ge\alpha^*$.

\emph{Rotations.} Let $G_M$ be Følner sets in $K$ and $\mu_M=|G_M|^{-1}\sum_{u\in G_M}(\rho_u)_*\mu_0$. Each $(\rho_u)_*\mu_0$ is translation invariant, because $(s_w)_*(\rho_u)_*\mu_0=(\rho_u)_*(s_{u^{-1}w})_*\mu_0=(\rho_u)_*\mu_0$; and $(\rho_u)_*\mu_0(0\in\omega)=\mu_0(0\in\omega)$ since $0\in u\omega\iff0\in\omega$. A weak-$*$ limit point $\mu$ of $(\mu_M)$ is translation invariant, rotation invariant by the Følner property, and has $\mu(0\in\omega)\ge\alpha^*$.
\end{proof}

\begin{remark}\label{rem:weights}
In~\cite[Prop.~4.2]{OAI} a single colouring is averaged over Følner sets of $E\rtimes K$. Applied to a fixed maximum independent set $J$ of a finite set and a finite averaging set $\Phi\subset E\rtimes K$, the empirical measure $\nu_\Phi=|\Phi|^{-1}\sum_{g\in\Phi}\delta_{gJ}$ has $\nu_\Phi(0\in\omega)=\sum_wc(w)\one_J(w)/\sum_wc(w)$ with multiplicities $c(w)=\#\{g\in\Phi:g^{-1}0=w\}$ that are not constant in $w$, i.e.\ the $c$-weighted proportion of $J$, for which no lower bound $\alpha^*$ is available a priori. Averaging first over translations (where each position is counted once) and then over rotations (which fix $0$) avoids this.
\end{remark}

\subsection{Projection onto the continuous factor}
Fix $\mu$ from Lemma~\ref{lem:inv}. Since $\Omega$ is compact metrizable, $(\Omega,\mu)$ is a standard probability space. Define unitaries on $L^2(\mu)$ by $T_zF(\omega)=F(\omega-z)$ and $V_uF(\omega)=F(u^{-1}\omega)$; then $T_zT_w=T_{z+w}$ and $V_uT_zV_u^{-1}=T_{uz}$. Let $f=\one[0\in\omega]$. Then $T_zf=\one[z\in\omega]$, $V_uf=f$, $\E f\ge\alpha^*$, and since $0$ and $u\in K$ are adjacent, $f\cdot T_uf=0$ pointwise, so
\begin{equation}\label{eq:prop}
  \langle f,T_uf\rangle=0\qquad(u\in K).
\end{equation}

Let $p=\E(f\mid\Fc)\in\Hc$ and $b=f-p\in\Hc^\perp$ (Lemma~\ref{lem:cont}). Then $0\le p\le1$ a.e., $\E p=\E f$, and $V_up=p$, $V_ub=b$. Both $p$ and $b$ are real-valued, since $f$ is.

Let $\lambda_b$ be the spectral measure of $b$ on $D$: $\langle b,T_zb\rangle=\int_D\psi_z\,d\lambda_b$. By Lemma~\ref{lem:cont}(ii), $\lambda_b(C)=0$. For $u\in K$,
\[
  \langle b,T_{uz}b\rangle=\langle b,V_uT_zV_u^{-1}b\rangle=\langle V_u^{-1}b,T_zV_u^{-1}b\rangle=\langle b,T_zb\rangle,
\]
With $m_u:D\to D$, $m_u(d)(z)=d(uz)$, we have $\int_D\psi_z\,d((m_u)_*\lambda_b)=\int_D\psi_{uz}\,d\lambda_b$, so $(m_u)_*\lambda_b$ and $\lambda_b$ have the same Fourier coefficients, and $\lambda_b$ is $K$-invariant. If $b\ne0$, Theorem~\ref{thm:rigidity} applied to $\lambda_b/\|b\|^2$ gives $\langle b,T_zb\rangle=0$ for $z\ne0$; this is trivial if $b=0$. Since $T_u$ preserves $\Hc$ and $\Hc^\perp$, $\langle p,T_ub\rangle=\langle b,T_up\rangle=0$. Expanding~\eqref{eq:prop},
\[
  \langle p,T_up\rangle=0\qquad(u\in K).
\]
By Lemma~\ref{lem:cont}(iii), $(T_x)_{x\in\R^2}$ is a strongly continuous unitary representation on $\Hc\ni p$, so $u\mapsto\langle p,T_up\rangle$ is continuous and, $K$ being dense in $S^1$, vanishes on $S^1$. Hence for all $x\in\R^2$ and $u\in S^1$,
\begin{equation}\label{eq:corr}
  \int(T_xp)(T_{x+u}p)\,d\mu=\langle T_xp,T_{x+u}p\rangle=\langle p,T_up\rangle=0 .
\end{equation}
Moreover $0\le T_xp\le1$ a.e.\ (apply positivity to $p$ and $\one-p$, using $T_x\one=\one$); in particular $T_xp$ is real-valued, and $\E[T_xp]=\langle T_xp,\one\rangle=\langle p,T_{-x}\one\rangle=\E p\ge\alpha^*$.

\subsection{A measurable field and one sample}

\begin{lemma}\label{lem:field}
Let $x\mapsto v_x$ be a continuous map from $\R^2$ to the real Hilbert space $L^2(\mu;\R)$. There is a jointly measurable $H:\R^2\times\Omega\to\R$ with $H(x,\cdot)=v_x$ $\mu$-a.e.\ for every $x\in\R^2$.
\end{lemma}

\begin{proof}
Fix a dense sequence $(x_k)$ in $\R^2$ and a representative of each $v_{x_k}$. For $j\ge1$ let $k_j(x)$ be the least $k$ with $\|v_{x_k}-v_x\|_2<2^{-j}$; it exists by continuity and density, and $x\mapsto k_j(x)$ is Borel because each $x\mapsto\|v_{x_k}-v_x\|_2$ is continuous. Put $H_j(x,\omega)=v_{x_{k_j(x)}}(\omega)=\sum_k\one[k_j(x)=k]\,v_{x_k}(\omega)$, which is jointly measurable as a countable sum of products of measurable functions. For fixed $x$, $\sum_j\|H_j(x,\cdot)-v_x\|_1\le\sum_j\|H_j(x,\cdot)-v_x\|_2<\sum_j2^{-j}<\infty$ ($\mu$ is a probability measure), so by Tonelli $\sum_j|H_j(x,\cdot)-v_x|<\infty$ a.e., and $H_j(x,\cdot)\to v_x$ a.e. Let $H=\limsup_jH_j$ where finite and $0$ elsewhere.
\end{proof}

Apply Lemma~\ref{lem:field} to the real-valued $v_x=T_xp$ and replace $H$ by $\min(\max(H,0),1)$. Fix the family $W$ (discs or squares) and integer radii $r\in\N$. By~\eqref{eq:corr} and Tonelli, for every $r\in\N$,
\[
  \int_\Omega\int_{B(0,r)}\int_{S^1}H(x,\omega)H(x+u,\omega)\,d\sigma(u)\,dx\,d\mu(\omega)=0,
\]
where $\sigma$ is normalized arc-length measure on $S^1$,
and the inner double integral is nonnegative, so it vanishes for $\mu$-a.e.\ $\omega$; intersecting over the countably many $r\in\N$ gives a $\mu$-conull set $\Omega_1$ on which it vanishes for every $r\in\N$. Let $X_r(\omega)=|W_r|^{-1}\int_{W_r}H(x,\omega)\,dx\in[0,1]$. By Tonelli, $\E X_r=\E p\ge\alpha^*$, and by the reverse Fatou lemma ($X_r\le1$), $\E[\limsup_rX_r]\ge\alpha^*$; hence $\mu(\limsup_rX_r\ge\alpha^*)>0$ (otherwise $\alpha^*-\limsup_rX_r>0$ a.e., so its integral would be positive). Fix $\omega\in\Omega_1$ with $\limsup_rX_r(\omega)\ge\alpha^*$ and set $h=H(\cdot,\omega):\R^2\to[0,1]$:
\begin{equation}\label{eq:sample}
  \int_{B(0,r)}\int_{S^1}h(x)h(x+u)\,d\sigma\,dx=0\ \ (r\in\N),\qquad
  \limsup_{r\in\N}\frac1{|W_r|}\int_{W_r}h\ \ge\ \alpha^*.
\end{equation}

\subsection{Density points}
Let $A_0=\{h>0\}$. By~\eqref{eq:sample} the nonnegative function $h(x)h(x+u)$ vanishes for $dx\otimes\sigma$-a.e.\ $(x,u)\in B(0,r)\times S^1$, for every $r$; since $\one_{A_0}(x)\one_{A_0}(x+u)>0$ iff $h(x)h(x+u)>0$, also $\int_{B(0,r)}\int_{S^1}\one_{A_0}(x)\one_{A_0}(x+u)\,d\sigma\,dx=0$ for all $r$, and since $h\le\one_{A_0}$, $\limsup_{r\in\N}|A_0\cap W_r|/|W_r|\ge\alpha^*$. Let $A$ be the set of Lebesgue density points of $A_0$. It is Borel, and $|A\triangle A_0|=0$ by the Lebesgue density theorem, so $A$ has the same densities as $A_0$.

\begin{lemma}\label{lem:dp}
$A$ contains no two points at distance $1$.
\end{lemma}

\begin{proof}
Suppose $x,\,y=x+u\in A$ with $|u|=1$. For $0<\eta<1$ let $x'$ be uniform on $B(x,\eta)$ and, independently, $u'$ uniform on the arc $\{e^{i\theta}u:|\theta|<\eta/2\}$, so $|u'-u|<\eta$. Then $\P(x'\notin A_0)=|B(x,\eta)\setminus A_0|/|B(x,\eta)|\to0$. Given $u'$, $x'+u'$ is uniform on $B(x+u',\eta)\subset B(y,2\eta)$, so $\P(x'+u'\notin A_0)\le|B(y,2\eta)\setminus A_0|/(\pi\eta^2)=4|B(y,2\eta)\setminus A_0|/|B(y,2\eta)|\to0$. So for small $\eta$, $\P(x'\in A_0,\ x'+u'\in A_0)>0$. The law of $(x',u')$ is absolutely continuous with respect to $dx\otimes\sigma$ and is supported in $B(0,r)\times S^1$ for an integer $r>|x|+1$. Hence the set $\{(x',u'):x'\in A_0,\ x'+u'\in A_0\}\cap(B(0,r)\times S^1)$ has positive $dx\otimes\sigma$-measure, contradicting $\int_{B(0,r)}\int_{S^1}\one_{A_0}(x)\one_{A_0}(x+u)\,d\sigma\,dx=0$.
\end{proof}

Thus $A$ is a measurable unit-distance-free set with $\limsup_{r\in\N}|A\cap W_r|/|W_r|\ge\alpha^*$, for $W$ = discs and for $W$ = squares. This proves $m_1^W\ge\alpha^*$ for all four conventions, completing Remark~\ref{rem:conventions} and, with Propositions~\ref{prop:conv} and~\ref{prop:LR}, Theorem~\ref{thm:main}. Corollary~\ref{cor:num} follows from~\cite{Croft,ACMVZ}.

\section{Formalization}\label{sec:lean}

The repository~\cite{repo} contains a Lean~4 project (toolchain v4.34.1, Mathlib \texttt{d13f23b}). The directory \texttt{OAI/} is a verbatim copy of the 70-file import closure of OpenAI's five-colour theorem from~\cite{OAIrepo} (Apache-2.0). Our files define
\begin{itemize}
\item \texttt{indepNum P} (independence number of the unit-distance graph on a \texttt{Finset} $\C$),
\item \texttt{f n}, the minimum of \texttt{indepNum P} over $n$-element \texttt{P},
\item \texttt{upperDensity} (squares $[-m,m]^2$, $m\in\N$) and \texttt{m1},
\end{itemize}
and prove \texttt{f\_div\_tendsto\_m1} ($f(n)/n\to m_1$) and \texttt{m1\_mul\_le\_f} ($m_1n\le f(n)$). The file \texttt{Disc.lean} defines \texttt{m1Ball} with open discs of real radius (the convention of~\cite{EP}) and proves \texttt{f\_div\_tendsto\_m1Ball}, \texttt{m1Ball\_mul\_le\_f} and \texttt{m1Ball\_eq\_m1}. Corollary~\ref{cor:num} is not formalized: its numerical bounds are quoted from~\cite{Croft,ACMVZ}. \texttt{\#print axioms} reports only \texttt{propext}, \texttt{Classical.choice} and \texttt{Quot.sound} for all of these and for \texttt{wild\_fourier}. The files in \texttt{OAI/} are compiled from source, not taken from a cache, and \texttt{lake env leanchecker Erdos1070} re-checks the whole import closure (Mathlib, \texttt{OAI/} and our files) in the Lean kernel. The Lean proof follows Sections~\ref{sec:conv}--\ref{sec:transfer}; the Følner sets in $E$ are constructed explicitly as above, and the weak-$*$ limits are realized by ultrafilter limits of positive functionals and the Riesz--Markov theorem.

\section*{Acknowledgements}
The negative answer to the second question of Erd\H{o}s is due to D\'ucz and Varga~\cite{DV}. Theorem~\ref{thm:rigidity} and the transfer template are from~\cite{OAI}. AI tools were used in this work: the observation that~\cite[Section~4]{OAI} applies to a single colour class, the write-up and the Lean formalization were produced with Claude (Anthropic), and the manuscript was audited with GPT (OpenAI).

\end{document}